\documentclass[12pt]{amsart}

\usepackage{setspace}
\usepackage{amsmath, amssymb,amsthm, amsfonts}
\usepackage{graphicx,color}
\newtheorem{thm}{Theorem}[section]

\theoremstyle{definition}
\newtheorem{definition}{Definition}

\theoremstyle{remark}
\newtheorem{remark}{Remark}[section]

\numberwithin{equation}{section}

\newcommand{\Ric}{\mbox{Ric}}
\newcommand{\R}{\mathbb R}

\newcommand{\be}{\begin{equation}}
\newcommand{\ee}{\end{equation}}
\newcommand{\bee}{\begin{equation*}}
\newcommand{\eee}{\end{equation*}}

\def\p{\partial}

\def\lf{\left}
\def\ri{\right}
\def\Pi{\displaystyle{\mathbb{II}}}

\def\m{\mathfrak{m}}

\def\a{\alpha}

\def\bci{\mathbf{c}_{_{I}}}
\def\ci{c_{_{I}}}
\def\ccs{c_{_{\rm CS} }}
\def\bccs{\mathbf{  c}_{_{  \rm CS} }}

\def\odd{\text{\rm odd}}
\def\even{\text{\rm even}}

\begin{document}

\title[]
{Evaluation of the ADM mass and center of mass via  the  Ricci tensor}

\author{Pengzi Miao$^1$}
\address[Pengzi Miao]{Department of Mathematics, University of Miami, Coral Gables, FL 33146, USA.}
\email{pengzim@math.miami.edu}
\thanks{$^1$Research partially supported by Simons Foundation Collaboration Grant for Mathematicians \#281105.}

\author{Luen-Fai Tam$^2$}
\address[Luen-Fai Tam]{The Institute of Mathematical Sciences and Department of
 Mathematics, The Chinese University of Hong Kong, Shatin, Hong Kong, China.}
 \email{lftam@math.cuhk.edu.hk}
\thanks{$^2$Research partially supported by Hong Kong RGC General Research Fund  \#CUHK 403108}

\renewcommand{\subjclassname}{
  \textup{2010} Mathematics Subject Classification}

\date{August 2014; revised in January 2015}

\begin{abstract} We prove directly without using  a density theorem  that
  (i) the  ADM mass defined in the usual way on an asymptotically flat manifold is equal to
 the mass defined intrinsically using the Ricci tensor;
 (ii) the Hamiltonian formulation of center of mass
   and the center of mass
 defined intrinsically using
 the
 Ricci tensor are the same.
\end{abstract}
 
\maketitle

\markboth{Pengzi Miao and Luen-Fai Tam}{Evaluation of the mass and center of mass}

\section{introduction}

Let $(M^n,g)$ be an end of some asymptotically flat manifold, i.e.
$M^n$ is diffeomorphic to $\R^n\setminus B(1)$, where $B(r)=\{|x|<r\}$, such that in the coordinates $\{ x^ i \}$ on $\R^n$,
$g_{ij}-\delta_{ij}$ together with its derivatives decays at infinity, which will be made precise later. We will use the Einstein summation convention throughout  this paper, i.e. summation over any pair of repeated indices.

For each large $r$, let
\be\label{eq-ADM-r-def}
\mathfrak{m}(r)=\frac{1}{2(n-1)\omega_{n-1}}\int_{S_r}(g_{ij,j}-g_{jj,i})\nu_e^i d\sigma_e ,
\ee
where $S_r=\p B(r)$, $\nu_e$ is the unit outward normal and $d\sigma_e$ is the area element on $S_r$ with respect to the Euclidean metric and $\omega_{n-1}$ is the area of the unit sphere in $\R^n$. The ADM mass \cite{ADM61} of $(M^n,g)$ is defined as
\be\label{eq-ADM-def}
\mathfrak{m}=\lim_{r\to\infty} \mathfrak{m}(r)
\ee
provided the limit exists. Under
suitable conditions,
it was proved by Bartnik \cite{Bartnik-1986}
 and Chru\'sciel \cite{Chrusciel} independently
 that $\mathfrak{m}$ is defined and does not  depend on the choice of  coordinates.

In the general relativity literature (see Ashtekar--Hansen \cite{AH-1978} and Chru\'{s}ciel \cite{Chrusciel-cqg}),
it is also known that the ADM mass  $ \m $ can   be  computed using the curvature of $g$ as follows\footnote{We  thank Piotr Chru\'{s}ciel  for bringing our attention the references \cite{AH-1978} and \cite{Chrusciel-cqg}.}:
Consider
\be\label{eq-Mass-i-r-def}
\mathfrak{m}_{_I}(r)=\frac{1}{(n-1)(2-n)\omega_{n-1}}\int_{S_r}\lf(\Ric -\frac12R_gg\ri)(X,\nu_g) d\sigma_g ,
\ee
where $\Ric$ and $R_g$ are the Ricci tensor and  the scalar curvature of $g$ respectively,
$X$ is the Euclidean conformal Killing vector field $x^i\frac{\p}{\p x^i}$,
$\nu_g$ is the unit outward normal and $d\sigma_g$ is the area element on $S_r$ with respect to  $g$.
If the limit
\be\label{eq-Mass-i-def}
\mathfrak{m}_{_I}=\lim_{r\to\infty}\mathfrak{m}_{_I}(r)
\ee
 exists, then
\be \label{eq-mass-equal}
\m = \m_{_I} .
\ee
Formula \eqref{eq-mass-equal} was also suggested by Schoen (cf. \cite{Huang-2009, CorvinoPollack})
in connection with  the generalized  Pohozaev identity \cite{Schoen}.
(For a recent application of \eqref{eq-mass-equal}
to the content at infinity of asymptotically flat   metrics,
see   \cite[Proposition 2.2]{Carlotto-Schoen-2014}.)
It can be easily checked that \eqref{eq-mass-equal} holds for metrics
which are conformally  flat  (up to higher order) near infinity.
Thus a common  proof of \eqref{eq-mass-equal} is to apply
a  density theorem from \cite{CorvinoSchoen2006} or \cite{SchoenYau81}
to reduce the general case to metrics  with { harmonic   asymptotics}.
An outline of such an approach
was given  in  \cite{Huang-2012}.

When
 $\m \neq 0$, there exist  several notions of center of mass for $(M, g)$ (cf. \cite{Huang-2009, Huang-2012}).
Similar to the definition of $\m$,
the Hamiltonian formulation of center of mass  $\bccs$,
 proposed by  Regge--Teitelboim \cite{RT-74} (also by Beig--\'{O} Murchadha \cite{BO-87} and Corvino--Schoen  \cite{CorvinoSchoen2006})
is given as follows:  let
 \bee \label{eq-CS-r}
\begin{split}
\ccs^\a(r)=&\frac1{2(n-1)\omega_{n-1} \mathfrak{m}}\int_{S_r}
\lf[ x^\a(g_{ij,i}-g_{ii,j})\nu_e^j-\lf(g_{i\a}\nu_e^i-g_{ii}\nu_e^\a\ri)\ri]d\sigma_e ,
\end{split}
\eee
where $\alpha = 1, \cdots, n$.
 Note that $g_{i\a}, g_{ii}$ in the second term in the integrand can be replaced by $h_{i\a}, h_{ii}$
where $h_{ij}:=g_{ij}-\delta_{ij}$. Let
$$\bccs(r)=(\ccs^1(r),\ccs^2(r),\dots,\ccs^n(r))  , $$
then
\be\label{eq-CS}
\bccs=\lim_{r\to\infty}\bccs(r)
\ee
provided the limit exists. (Here we use the notation $\bccs$ following  \cite{Huang-2009}.)
Similar to  \eqref{eq-Mass-i-def},  Schoen suggested an intrinsic way to define
the center of mass   (cf. \cite{Huang-2009}):
for $\a=1,\dots,n$, let $Y_{(\a)}$ be the Euclidean conformal Killing vector field
$$\lf(|x|^2\delta^{\a i}-2x^\a x^i\ri)\frac{\p}{\p x^i},$$
define
\be\label{eq-I-r}
\ci^\a(r)=\frac{1}{2(n-1)(n-2)\omega_{n-1}\mathfrak{m}}\int_{S_r}\lf(\Ric -\frac12R_gg\ri)(Y_{(\a)},\nu_g) d\sigma_g
\ee
and
$$\bci(r)=(\ci^1(r),\dots,\ci^n(r)) .$$
The intrinsic  center of mass is defined as
\be\label{eq-I}
\bci=\lim_{r\to\infty}\bci(r)
\ee
provided the limit exists.
We also want to mention that Huisken-Yau \cite{Huisken-Yau1996} and Ye \cite{Ye1996}
constructed a foliation  of stable  constant mean curvature spheres near infinity  on asymptotically
Schwarzschild manifolds via different methods.
Huisken-Yau   \cite{Huisken-Yau1996} proposed a geometric
definition of center of mass using the foliation.
It was proved by Huang  \cite{Huang-2009} that
under  the  Regge--Teitelboim  condition (see Theorem \ref{t-main} (b)),
all these  notions of center of mass are equivalent.
In order to show $\bccs = \bci$, in \cite{Huang-2009} Huang first proved a density theorem
for metrics satisfying  the  Regge--Teitelboim  condition
and then apply it to  reduce  the general  case to metrics with
  harmonic   asymptotics.
  Recently, in  \cite{Nerz2014}, Nerz also  used $\m_{_I}$ as
  the definition of mass in his  construction of foliations of constant mean curvature spheres for asymptotically flat manifolds
  under weaker   asymptotic conditions.

In this paper, we give a direct proof of $ \bccs = \bci $ and $ \m = \m_{_I}$
without using density theorems. More precise, we will prove the following:

\begin{thm}\label{t-main}
On an asymptotically flat end $(M^n, g)$ which satisfies
$g_{ij}-\delta_{ij}=o_2(|x|^{-\frac{n-2}2})$,
 one has
\begin{enumerate}
   \item [(a)]
\bee
\lim_{r\to\infty}\lf(\mathfrak{m}(r)-\mathfrak{m}_{_I}(r)\ri)=0.
\eee
   \item [(b)] If  the  Regge--Teitelboim condition holds, i.e. $g_{ij}^{\text{\rm odd}}(x)= o_2(|x|^{-\frac n2})$, then
\bee
\lim_{r\to\infty}\lf(\bccs(r)-\bci(r)\ri)=0.
\eee
 \end{enumerate}
\end{thm}

As mentioned earlier, part  (a) of Theorem \ref{t-main} is known to experts in the relativity  community;
(b) was first proved by Huang \cite{Huang-2009} (under slightly different decay assumptions).
Our contribution   is to provide more elementary and simpler proofs.

The structure of the paper is as follows. In the next section,  we prove part (a) of Theorem \ref{t-main}.
In section 3, we prove  part  (b).

  {\it Acknowledgement:} The authors would like to thank Richard Schoen for some useful discussions.

\section{$\mathfrak{m}=\mathfrak{m}_{_I}$}

On  an asymptotically flat end $(M^n, g)$,
as  our discussion is only  near  the infinity, we  may extend $(M^n, g)$
so that $M$ is diffeomorphic to $\R^n$.
 We will assume this throughout  the rest of  the  paper.

 \begin{definition}
 Let $f$ be a function defined near infinity of $\R^n$. We say that $f=o_k(|x|^{-\tau})$, if $f$ is in $C^k$ and $|x|^{|\a|+\tau}|\p ^\a f(x)|=o(1)$ as $x\to\infty$ for all $\a$ with $0\le |\a|\le k$.
 \end{definition}

\begin{thm}\label{t-ADM}  Suppose
$g_{ij}-\delta_{ij}=o_2(|x|^{-\frac{n-2}2})$ on $(M^n, g)$. 
Then
\bee
\lim_{r\to\infty}\lf(\mathfrak{m}(r)-\mathfrak{m}_{_I}(r)\ri)=0.
\eee
In fact, the following is true. Let   $ \{ D_l \}_{l=1}^\infty $ be a sequence of  bounded open sets with Lipschitz boundary 
  $\Sigma_l:=\p D_l$ which has  area $|\Sigma_l|$. Let $r_l:=\inf_{x\in \Sigma_l}|x|$. Assume
$$
\lim_{ l\to\infty}r_l=\infty \ \ {\text and}\ \  |\Sigma_l|\le Cr_l^{n-1}
$$
for some constant $C$ independent of $l$. Then
\bee
\lim_{ l \to\infty}\lf(\mathfrak{m}(l)-\mathfrak{m}_{_I}(l)\ri)=0,
\eee
where $\mathfrak{m}(l)$ is the RHS in \eqref{eq-ADM-r-def} integrating over $\Sigma_l$ and $\mathfrak{m}_{_I}(l)$ is the RHS in \eqref{eq-Mass-i-r-def}  integrating over $\Sigma_l$.
\end{thm}
\begin{proof}
The condition $g_{ij}-\delta_{ij}=o_2(|x|^{-\frac{n-2}2})$ shows
\be \label{eq-normal-compare}
 |\nu^i_e-\nu_g^i|=o(r_l^{-\frac{n-2}2}), \ \ d\sigma_g=\lf(1+o(r_l^{-\frac{n-2}2})\ri)d\sigma_e
 \ee
on $\Sigma_l$, and
\be\label{eq-Ricci}
\begin{split}
2R_{ij}(x)= & \ 2\frac{\p}{\p x^k}\Gamma^k_{ji}-2\frac{\p}{\p x^j}\Gamma^k_{ki}+
2\Gamma^k_{kl}\Gamma^l_{ji}-2\Gamma^k_{jl}\Gamma^l_{ki}\\
= & \ \frac{\p}{\p x^k}\lf(g_{ki,j}+g_{kj,i}-g_{ij,k}\ri)-\frac{\p}{\p x^j}\lf(g_{ki,k}+g_{kk,i}-g_{ki,k}\ri)+o(|x|^{-n})\\
= & \ g_{ki,kj}+g_{kj,ki}-g_{ij,kk}-g_{kk,ij}+o(|x|^{-n})\\
= &  \ o(|x|^{-\frac{n+2}2}).
\end{split}
\ee

Using the fact that $|\Sigma_l|\le Cr_l^{n-1}$, by \eqref{eq-Ricci} as $l\to\infty$ we have:
\be\label{eq-mass-1i}
\begin{split}
 -2\int_{\Sigma_l}  R_{ij}x^i\nu_g^j  d\sigma_g
=&-2\int_{\Sigma_l}R_{ij}x^i\nu_e^j d\sigma_g-2\int_{\Sigma_l}R_{ij}x^i(\nu_g^j-\nu_e^j) d\sigma_g\\
=&-2\int_{\Sigma_l}R_{ij}x^i\nu_e^j d\sigma_e+o(1)\\
=&\int_{\Sigma_l}\lf(-g_{ki,kj}-g_{kj,ki}+g_{ij,kk}+g_{kk,ij}\ri)x^i\nu^j_e d\sigma_e+o(1)  .
\end{split}
\ee
We claim that
 \be\label{eq-mass-1ii}
 \begin{split}
&\int_{\Sigma_l}\lf(-g_{ki,kj}-g_{kj,ki}+g_{ij,kk}+g_{kk,ij}\ri)x^i\nu^j_e d\sigma_e\\
= &  (n-2)\int_{\Sigma_l} \lf(g_{kj,k }-g_{kk,j }\ri)     \nu_e^j d\sigma_e   +\int_{\Sigma_l}  \lf(-g_{kj,k j}+g_{kk, jj}\ri)x^i\nu^i_e d\sigma_e.
\end{split}
\ee
If the claim is true, then by   \eqref{eq-normal-compare} and \eqref{eq-mass-1i}, we have
\be\label{eq-mass-1}
\begin{split}
&-2\int_{\Sigma_l}  R_{ij}x^i\nu_g^j  d\sigma_g\\
=&(n-2)\int_{\Sigma_l} \lf(g_{kj,k }-g_{kk,j }\ri)\nu_g^j d\sigma_g+\int_{\Sigma_l}  \lf(-g_{kj,k j}+g_{kk, jj}\ri)x^i\nu^i_e d\sigma_e+o(1).
\end{split}
\ee
 To verify \eqref{eq-mass-1ii},
viewing each $g_{ij}$   as functions on  $\R^n$, we may find sequences of smooth functions $ \{ g_{ij}^{(m)} \} $ such that $g_{ij}^{(m)}=g_{ji}^{(m)}$
for all  $m$ and  $ \{ g_{ij}^{(m)} \}$ converges   to $g_{ij}$   uniformly in $C^2$ norm  on any compact sets.
Hence to prove \eqref{eq-mass-1ii}, we may assume that $g$ is $C^3$. 
 Integrating by parts on $D_l$ (since $\p D_l$ is Lipschitz), we have
\bee
\begin{split}
&\int_{\Sigma_l}\lf(-g_{ki,kj}-g_{kj,ki}+g_{ij,kk}+g_{kk,ij}\ri)x^i\nu^j_e d\sigma_e\\
 &\int_{D_l}\frac{\p}{\p x^j}\lf[\lf(-g_{ki,kj}-g_{kj,ki}+g_{ij,kk}+g_{kk,ij}\ri)x^i\ri] dv_e \\
=&\int_{D_l} \lf(-g_{ki,kj}-g_{kj,ki}+g_{ij,kk}+g_{kk,ij}\ri)\delta^i_j  dv_e\\
&+\int_{D_l} \lf(-g_{ki,kjj}-g_{kj,kij}+g_{ij,kkj}+g_{kk,ijj}\ri)x^i dv_e \\
=&2\int_{D_l} \lf(g_{kk,jj}-g_{kj,kj}\ri) dv_e
 +\int_{D_l}\lf(-g_{kj,kij}+g_{kk,ijj}\ri)x^i dv_e\\
 &\quad(\text{\bf since $\sum_{k,j}(-g_{ki,kjj}+g_{ij,kkj})=0$})\\
=&(n-2)\int_{D_l} \lf(g_{kj,kj}-g_{kk,jj}\ri) dv_e
 +\int_{D_l} \frac{\p}{\p x^i}\lf(\lf(-g_{kj,k j}+g_{kk, jj}\ri)x^i\ri)dv_e   \\
=&(n-2)\int_{\Sigma_l} \lf(g_{kj,k }-g_{kk,j }\ri)\nu_e^j d\sigma_e+\int_{\Sigma_l}  \lf(-g_{kj,k j}+g_{kk, jj}\ri)x^i\nu^i_e d\sigma_e ,
\end{split}
\eee
where $dv_e$ is the volume element with respect to the Euclidean metric.
 This proves \eqref{eq-mass-1ii}.

On the other hand,  by \eqref{eq-Ricci}, we have
\be\label{eq-mass-2}
\begin{split}
R_g(x)
=&\sum_i R_{ii}+o(|x|^{-n})\\
=&\frac12\sum_{i,k}
\lf(g_{ki,ki}+g_{ki,ki}-g_{ii,kk}-g_{kk,ii}\ri)+o(|x|^{-n})\\
=& g_{ik,ik}-g_{kk,ii} +o(|x|^{-n})\\
=&o(|x|^{-\frac{n+2}2}) .
\end{split}
\ee
Hence
\be\label{eq-mass-3}
\begin{split}
\int_{\Sigma_l}  \lf(-g_{kj,k j}+g_{kk, jj}\ri)x^i\nu^i_e d\sigma_e=&-\int_{\Sigma_l}  R_gx^i\nu^i_e d\sigma_e+o(1)\\
=&-\int_{\Sigma_l}  R_gg(X,\nu_g)  d\sigma_g+o(1) .
\end{split}
\ee
  Combining  this with  \eqref{eq-mass-1}, we conclude that
\be
-2\int_{\Sigma_l}  \lf(\Ric-\frac12 R_g g\ri)(X, \nu_g)  d\sigma_g=(n-2)\int_{\Sigma_l} \lf(g_{kj,k }-g_{kk,j }\ri)\nu_g^j d\sigma_g+o(1)
\ee
as $l\to\infty$. From this it is easy to see the theorem is true.
\end{proof}

\section{$\bccs=\bci$}

\begin{definition}
For a function $f(x)$ defined on $\R^n$, let
$$f^{\text{\rm odd}}(x):=\frac12( f(x)-f(-x)),\ \ f^{\text{\rm even}}(x):=\frac12( f(x)+f(-x)).$$
\end{definition}

\begin{thm}\label{t-CS}
Suppose
$g_{ij}-\delta_{ij}=o_2(|x|^{-\frac{n-2}2})$ on $(M^n, g)$.
Suppose $ g$ also satisfies the Regge--Teitelboim condition:
$g_{ij}^{\text{\rm odd}}(x)= o_2(|x|^{-\frac n2})$. Then
\bee
\lim_{r\to\infty}\lf(\bccs(r)-\bci(r)\ri)=0.
\eee
\end{thm}
\begin{proof} For $x\in S_r$, we have the following:
 \be\label{eq-cm-1}
\left\{
  \begin{array}{ll}
    g_{ij}(x)=& \lf(g_{ij}\ri)^{\even}(x)+o(|x|^{-\frac n2}); \\
     \frac{\p g_{ij}}{\p x^k}(x)=&\lf(\frac{\p g_{ij}}{\p x^k} \ri)^{\odd}(x) +o(|x|^{-1-\frac n2}); \\
    \Gamma_{ij}^k(x)=&\lf(\Gamma_{ij}^k\ri)^{\odd}(x)  +o(|x|^{-1-\frac n2}) ; \\
    R_{ij}(x)=&\lf(R_{ij}\ri)^{\even}(x)+o(|x|^{-2-\frac n2})    ;\\
R_g(x)=&  (R_g)^{\even}(x)+o(|x|^{-2-\frac n2}).\\
  \end{array}
\right.
\ee
We also have
\begin{align}
\nu_g^i(x)-\nu_e^i(x) & = \lf(\nu_g^i -\nu_e^i \ri)^{\odd}(x)+o(|x|^{-\frac n2}) \label{eq-cm-2} , \\
f(x)-1 & =(f-1)^{\even}(x)+o(|x|^{-\frac n2}) \label{eq-cm-3} ,
\end{align}
where $f(x)$ is defined by    $d\sigma_g(x)=f(x)d\sigma_e(x)$.
For  each $\a=1,\dots,n$,
  define  $Y_{(\a)}(x)= Y_{(\a)}^i \frac{\p}{\p x^i}$, where
$Y_{(\a)}^i=\lf( |x|^2\delta^{\a i}-2x^\a x^i \ri)$.
Note that $Y^i(x)=Y^i(-x)$. By \eqref{eq-cm-1}--\eqref{eq-cm-3},
we have  as $r\to\infty$
\be\label{eq-cm-4}
\begin{split}
\int_{S_r}\Ric(Y,\nu_g)d\sigma_g=&\int_{S_r}\Ric(Y,\nu_g-\nu_e)(f-1)d\sigma_e
+\int_{S_r}\Ric(Y,\nu_g-\nu_e)  d\sigma_e\\
&+\int_{S_r}\Ric(Y, \nu_e)(f-1) d\sigma_e+\int_{S_r}\Ric(Y, \nu_e)  d\sigma_e
\\
=&\int_{S_r}\Ric(Y,\nu_e)d\sigma_e+o(1),
\end{split}
\ee
where we  also used the fact $g_{ij}(x)-\delta_{ij}=o_2(|x|^{-\frac{n-2}2})$.
Using \eqref{eq-cm-1} and \eqref{eq-Ricci}, as $r\to\infty$, we then have
\be\label{eq-cm-5}
\begin{split}
&2\int_{S_r}R_{ij}Y^i\nu_e^jd\sigma_e\\
=&2\int_{S_r}\lf(\frac{\p}{\p x^k}\Gamma^k_{ji}-\frac{\p}{\p x^j}\Gamma^k_{ki}\ri)Y^i\nu_e^jd\sigma_e+o(1)\\
=&\int_{S_r}\lf[\lf(g^{ks}(g_{is,jk}+g_{js,ik}-g_{ij,sk}\ri)
-g^{ks}\lf(g_{ks,ij}+g_{is,kj}-g_{ki,sj}\ri)\ri]
Y^i\nu_e^jd\sigma_e+o(1)\\
=&\int_{S_r} \lf( g_{ik,jk}+g_{jk,ik}-g_{ij,kk} - g_{kk,ij}\ri)
Y^i\nu_e^jd\sigma_e+o(1).
\end{split}
\ee
  As in the proof of Theorem \ref{t-ADM}, we may assume that $g$ is smooth to obtain: 
\bee
\begin{split}
 &\int_{S_r}\lf(g_{ki,kj}+g_{kj,ki}-g_{ij,kk}-g_{kk,ij}\ri) Y^i\nu_e^jd\sigma_e\\
=&\int_{B(r)}\frac{\p}{\p x^j}\lf[\lf(g_{ki,kj}+g_{kj,ki}-g_{ij,kk}-g_{kk,ij}\ri) Y^i\ri]dv_e\\
=&\int_{B(r)}\lf(g_{kj,kij} -g_{kk,ijj}\ri)Y^idv_e+\int_{B(r)}\lf(g_{ki,kj}+g_{kj,ki}-g_{ij,kk}-g_{kk,ij}\ri)\frac{\p}{\p x^j} Y^i dv_e\\
=&\int_{S_r}\lf(g_{kj,kj} -g_{kk,jj}\ri) Y^i\nu_e^id\sigma_e-\int_{B(r)} (g_{kj,kj} -g_{kk,jj})\frac{\p}{\p x^i}Y^idv_e\\
&+\int_{B(r)}\lf(g_{ki,kj}+g_{kj,ki}-g_{ij,kk}-g_{kk,ij}\ri)\frac{\p}{\p x^j}Y^idv_e.
\end{split}
\eee
Since
$$
\frac{\p}{\p x^j}Y^i=\frac{\p}{\p x^j}\lf(|x|^2\delta^{\a i}-2x^\a x^i\ri) =2x^j\delta^{\a i}-2\delta^{\a j}x^i-2x^\a \delta^i_j ,
$$
we have
\bee
\begin{split}
(g_{kj,kj} -g_{kk,jj})\frac{\p}{\p x^i}Y^i =-2nx^\a(g_{kj,kj}-g_{kk,jj})
\end{split}
\eee
and
\bee
\begin{split}
 \lf(g_{ki,kj}+g_{kj,ki}-g_{ij,kk}-g_{kk,ij}\ri)\frac{\p}{\p x^j}Y^i
=&-2x^\a \lf(g_{ki,ki}+g_{ki,ki}-g_{ii,kk}-g_{kk,ii}\ri)\\
=&-4x^\a\lf(g_{ki,ki}-g_{kk,ii}\ri).
\end{split}
\eee
Hence,
\be
\begin{split}
&\int_{S_r}\lf(g_{ki,kj}+g_{kj,ki}-g_{ij,kk}-g_{kk,ij}\ri) Y^i\nu_e^jd\sigma_e\\
=&\int_{S_r}\lf(g_{kj,kj} -g_{kk,jj}\ri) Y^i\nu_e^jd\sigma_e+2(n-2)\int_{B(r)}x^\a\lf(g_{ki,ki}-g_{kk,ii}\ri)dv_e\\
=&\int_{S_r}\lf(g_{kj,kj} -g_{kk,jj}\ri) Y^i\nu_e^id\sigma_e+2(n-2)\int_{S_r}x^\a(g_{ki,k}-g_{kk,i})\nu_e^i   d \sigma_e \\
&-2(n-2)\int_{B(r)}  (g_{k\a,k}-g_{kk,\a})dv_e\\
=&\int_{S_r}\lf(g_{kj,kj} -g_{kk,jj}\ri) Y^i\nu_e^id\sigma_e+2(n-2)\int_{S_r}\lf[x^\a\lf(g_{ki,k}-g_{kk,i}\ri)\nu_e^i- \lf(g_{k\a}\nu_e^k-g_{kk}\nu_e^\a\ri)\ri]d\sigma_e.
\end{split}
\ee
 Here  $g_{k\a}, g_{kk}$ in the second term in the   integrand of the second integral can be replaced by $h_{k\a}, h_{kk}$ where $h_{ij}:=g_{ij}-\delta_{ij}$.

Using \eqref{eq-cm-1}--\eqref{eq-cm-3}, we may argue as before to conclude that
\bee
\int_{S_r}\lf(g_{kj,kj} -g_{kk,jj}\ri) Y^i\nu_e^id\sigma_e=\int_{S_r}R_g g(Y,\nu_g) d\sigma_g+o(1),
\eee
and
\bee
\begin{split}
\int_{S_r}&\lf[x^\a\lf(g_{ki,k}-g_{kk,i}\ri)\nu_e^i- \lf(g_{k\a}\nu_e^k-g_{kk}\nu_e^\a\ri)\ri]d\sigma_e\\
=&\int_{S_r}\lf[x^\a\lf(g_{ki,k}-g_{kk,i}\ri)\nu_g^i- \lf(g_{k\a}\nu_g^k-g_{kk}\nu_g^\a\ri)\ri]d\sigma_g+o(1)
\end{split}
\eee
as $r\to\infty$.
Combining these with \eqref{eq-cm-4} and \eqref{eq-cm-5}, and using   \eqref{eq-mass-2}  for the expression of $R_g$, we have

\be
\begin{split}
\int_{S_r}&\lf(R_{ij}-\frac12R_g g_{ij}\ri)Y^i \nu_g^j d\sigma_g\\
=& (n-2)\int_{S_r}\lf[x^\a(g_{ki,k}-g_{kk,i})\nu_e^id\sigma_e-  \lf(g_{k\a}\nu_e^k-g_{kk}\nu^\a\ri)\ri]d\sigma_e+o(1) \\
=&(n-2)\int_{S_r}\lf[x^\a(g_{ki,k}-g_{kk,i})\nu_g^i -  \lf(g_{k\a}\nu_g^k-g_{kk}\nu_g^\a\ri)\ri]d\sigma_g+o(1)\\
\end{split}
\ee
as $r\to\infty$. From this the result follows.
\end{proof}

\begin{remark} On an
asymptotically flat $(M^n, g)$
with $g_{ij} -\delta_{ij} =O_2(|x|^{-q})$, where $q>\frac{n-2}{2}$,
it follows from the proof of  Theorem \ref{t-ADM}
in Section 2
that
$\mathfrak{m}$ and hence  $\mathfrak{m}_{_I}$ are  defined if and only if
$
\lim_{r\to\infty}\int_{B(r)}R_g dv_g
$
exists. On the other hand,  if
in addition
 $g^{\odd}=O_2(|x|^{-q-1})$,
the computation in \cite{CorvinoSchoen2006} (also cf. \cite{ChanTam})
shows
 $\bccs$ and hence  $\bci$ are  defined if and only if
$
\lim_{r\to\infty}\int_{B(r)}x^i R_g dv_g
$
exists  for   $ i =1,\dots,n$.
Here $f= O_2(|x|^{-q})$ means   $|x|^{|\a|+q}|\p^\a f|\le C$ for all $\alpha$ with  $|\a|\le 2$.
\end{remark}


\begin{thebibliography}{10}

\bibitem{ADM61} Arnowitt,  R.;  Deser, S., and Misner,  C. W.,
{\sl Coordinate invariance and energy expressions in general relativity}, Phys. Rev. (2) \textbf{122} (1961), 997--1006.

\bibitem{AH-1978}  Ashtekar,  A.;  Hansen, R. O.,
{\sl A unified treatment of null and spatial infinity in general relativity. I. Universal structure, asymptotic symmetries, and conserved quantities at spatial infinity},  J. Math. Phys.   \textbf{19}  (1978),    1542--1566.

\bibitem{Bartnik-1986} Bartnik, R., {\sl The mass of an asymptotically flat manifold}, Comm. Pure Appl. Math.  \textbf{39} (1986), no. 5, 661--693.


\bibitem{BO-87}  Beig, R.; \'{O} Murchadha, N.,
{\sl The Poincar\'{e} group as the symmetry group of canonical general relativity}, Ann. Physics \textbf{174} (1987), no. 2, 463--498.


\bibitem{Carlotto-Schoen-2014}   Carlotto, A.; Schoen, R.,
{\sl  Localizing solutions of the Einstein constraint equations}, arXiv: 1407.4766.

 \bibitem{ChanTam}   Chan, P.-Y.;   Tam, L.-F., {\sl A note on center of mass},  arXiv:1402.1220.

\bibitem{CorvinoPollack} Corvino, J.; Pollack, D., {\sl Scalar curvature and the Einstein constraint equations},
Surveys in geometric analysis and  relativity,  Adv. Lect. Math. (ALM), \textbf{20} (2011), 145--188,   Int. Press, Somerville, MA.

\bibitem{CorvinoSchoen2006} Corvino, J.; Schoen, R. M., {\sl On the asymptotics for the vacuum Einstein constraint equations}, J. Differential Geom. \textbf{73}  (2006),  no. 2, 185--217.


\bibitem{Chrusciel}
Chru\'sciel,  P.,  {\sl Boundary conditions at spatial   infinity from a Hamiltonian point of view},
Topological Properties and Global Structure of Space-Time, Plenum Press, New York, (1986), 49--59.


\bibitem{Chrusciel-cqg}
Chru\'sciel,  P.,  {\sl A remark on the positive-energy theorem},
Class. Quantum Grav. \textbf{3} (1986), L115--L121.




\bibitem{Huang-2009}  Huang, L.-H., {\sl On the center of mass of isolated systems with general asymptotics},
Class. Quantum Grav. \textbf{26}  (2009),  no. 1, 015012, 25 pp.


\bibitem{Huang-2012} Huang, L.-H., {\sl On the center of mass in general relativity}, Fifth International Congress of Chinese Mathematicians. Part 1, 2,  575–591, AMS/IP Stud. Adv. Math., 51, pt. 1, 2, Amer. Math. Soc., Providence, RI, 2012.

\bibitem{Huisken-Yau1996} Huisken, G.; Yau, S.-T., {\sl Definitions of center of mass for isolated physical systems and unique foliations by stable
spheres with constant mean curvature}, Invent. math., \textbf{124} (1996), 281--311.

\bibitem{RT-74} Regge, T.; Teitelboim, C.,
{\sl Role of surface integrals in the Hamiltonian formulation of general relativity},
Ann. Physics, \textbf{88} (1974), 286--318.

\bibitem{Nerz2014} Nerz, C.,
{\sl Foliations by stable spheres with constant mean curvature for isolated systems without asymptotic symmetry}, arXiv:1408.0752.


\bibitem{Schoen}
Schoen, R.,  {\sl The existence of weak solutions with prescribed singular behavior for a conformally invariant scalar equation},
Comm. Pure and Appl. Math. \textbf{41} (1988), 317--392.



\bibitem{SchoenYau81} Schoen, R.; Yau, S. T., {\sl The energy and the linear momentum of space-times in general
relativity}, Comm. Math. Phys. \textbf{79} (1981), no. 1, 47--51.





\bibitem{Ye1996} Ye, R., {\sl Foliation by constant mean curvature spheres on asymptotically flat manifolds}, in `Geometric analysis and the calculus of variations',  369--383, Int. Press, Cambridge, MA, 1996.


\end{thebibliography}
\end{document}